\documentclass[12pt]{amsart}

\usepackage{amssymb}
\usepackage{amsmath}
\usepackage{amsthm}
\usepackage{color}
\usepackage{mathrsfs}

\usepackage{amsfonts}
\usepackage{amssymb}
\usepackage{amsmath}
\usepackage{amsthm}
\usepackage{bbm}
\usepackage{verbatim}
\usepackage{url}

\allowdisplaybreaks

\newcommand{\eps}[0]{\varepsilon}

\def\supp{{\mathop\mathrm{\,supp\,}}}

\newtheorem{thm}{Theorem}[section]
\newtheorem{prop}[thm]{Proposition}

\newcommand{\R}{\mathbb R}

\newcommand{\Sph}{\mathbb{S}}
\newcommand{\card}{\operatorname{card}}
\newcommand{\wh}{\widehat}
\newcommand{\Haus}{\mathcal{H}}

\numberwithin{equation}{section}

\begin{document}
\arraycolsep=1pt

\title[Mobile sampling]{A sharp sufficient condition for mobile sampling in terms of surface density}
\author{Benjamin Jaye}
\author{Mishko Mitkovski}
\author{Manasa N. Vempati}


\maketitle

\begin{abstract}We provide the sharp surface density threshold to guarantee mobile sampling in terms of the surface density of the set.
\end{abstract}

\section{Introduction}

 For a set $\Gamma\subset \R^d$ of locally finite $(d-1)$-dimensional Hausdorff $\mathcal{H}^{d-1}$-measure, the mobile sampling problem concerns whether there exists a constant $C>0$ such that
$$\|f\|_{L^2(\R^d)}^2\leq C\int_{\Gamma}|f|^2d\mathcal{H}^{d-1}$$
for every function in $L^2(\R^d)$ whose Fourier transform
$$\widehat{f}(\xi) = \int_{\R^d}f(x)e^{2\pi i \langle x,\xi\rangle}dm_d(x)
$$is supported in an origin symmetric convex set $K$.

This problem has been quite heavily studied in the last ten years, see e.g. \cite{AGR, AGH+, GRUV, JNR, RUZ} and references therein, following foundational work by Unnikrishnan and Vetterli~\cite{UV12, UV13}, who formulated the problem precisely and coined the term mobile sampling. These papers contain a number of precise results characterizing mobile sampling sets within a variety of special families of curves and surfaces. Many of these characterizations use the concept of surface density also introduced by Unnikrishnan and Vetterli as an analog of lower Beurling density that featured in classical results of Beurling and Kahane. The lower surface density $\mathbf{D}^{-}(\Gamma)$ of a $(d-1)-$dimensional surface $\Gamma\subset \mathbb{R}^d$ is defined by 

\begin{equation*}
    \mathbf{D}^{-}(\Gamma) = \liminf_{r\mapsto \infty}\inf_{x\in \mathbb{R}^d}\frac{\mathcal{H}^{d-1}(\Gamma\cap B(x,r))}{m_d(B(x,r))}.
\end{equation*}



The goal of \cite{JM} was to provide a general sufficient condition for mobile sampling in terms of the lower surface density of $\Gamma$ that is valid for a large class of surfaces, in the spirit or the well-known one-dimensional results of Beurling and Kahane~\cite{Beu, Kah}. In \cite{JM}, it was shown that there is a constant $A_d$ such that any "regular'' surface $\Gamma \subset \mathbb{R}^d$ with surface density $\mathbf{D}^{-}(\Gamma) > A_d \mathbf{W}(K)$, must be a mobile sampling set. Here $\mathbf{W}(K)$ denotes the mean width of the symmetric convex set $K$. 

In \cite{JM} it was shown that one can take $A_d = \frac{\omega_d}{\omega_{d-1}}\frac{3d^2}{2d+4}$ -- where $\omega_d$ is the volume of the $d$-dimensional unit ball -- and it was left as an open problem if this constant could be improved. In this paper resolve this issue and provide the sharp value of the constant $A_d = \frac{d}{2}\frac{\omega_d}{\omega_{d-1}}$.  All relevant definitions will be given in the next section.

\begin{thm}\label{Maintheo}
    Set 
    \begin{equation*}
        A_d = \frac{d}{2}\frac{\omega_d}{\omega_{d-1}}
    \end{equation*}

Suppose that $\Gamma$ is $\varphi-$regular and 
\begin{equation*}
     D^{-}(\Gamma)> \varphi(0)\cdot A_d\cdot\mathbf{W}(K).
\end{equation*}
For every $1\leq p\leq \infty$, there exists a constant $C>0$ such that 
\begin{equation}\label{sampling ineq}
 \bigg(\int_{\mathbb{R}^d}|f|^pdm_d\bigg)^{1/ p} \leq C\bigg(\int_{\Gamma}|f|^pd\mathcal{H}^{d-1}\bigg)^{1/ p} 
\end{equation}
for every $f\in L^p(\R^d)$ whose (distributional) Fourier transform is supported in $K$.

\end{thm}

In Section 5 of \cite{JM} it was shown that $A_d \geq \frac{d}{2}\frac{\omega_d}{\omega_{d-1}}$ when $K=[-1,1]^d$. This shows that the constant $A_d$ given by Theorem~\ref{Maintheo} is sharp in this most classical case. It is plausible that this constant is sharp for a larger class of origin symmetric convex sets $K$.

In this paper we show that the constant $A_d = \frac{d}{2}\frac{\omega_d}{\omega_{d-1}}$ is also sharp for the Euclidean ball $K=B(0,1)$ in any dimension.

\begin{thm}\label{sharp}  For every $\delta>0$, there is a function $f$ with $f(0)=\|f\|_{\infty}=1$, $\operatorname{supp}(\widehat{f})\subset B(0,1)$, and a $\varphi$-regular set $\Gamma$ with $\varphi(0)=1$, $\Gamma \subset \{f=0\}$, and 
$$\mathbf{D}^-(\Gamma)>2\Bigl(\frac{d}{2}\frac{\omega_d}{\omega_{d-1}}-\delta\Bigl).$$
\end{thm}

Finally, we also show that for $d=2$ (in which case $A_2=\pi/2$), our constant is sharp for any convex set $K$ which is $\pi/2$-symmetric, meaning that $$(x_1,x_2)\in K \iff (-x_2,x_1)\in K.$$  This class of symmetric convex sets contains all $\ell^p$-balls for $1\leq p\leq \infty$. 

\begin{thm}\label{sharp2}  Suppose $d=2$, and $K$ is a $\pi/2$-symmetric convex set.  For every $\delta>0$, there is a function $f$ with $f(0)=\|f\|_{\infty}=1$, $\operatorname{supp}(\widehat{f})\subset K$, and a $\varphi$-regular set $\Gamma$ with $\varphi(0)=1$, $\Gamma \subset \{f=0\}$, and 
$$\mathbf{D}^-(\Gamma)>\Bigl(\frac{\pi}{2}-\delta\Bigl)\mathbf{W}(K).$$
\end{thm}

In order to prove Theorem \ref{Maintheo} we prove an improved bound on the density of the zero set of a Paley-Weiner class function (Proposition \ref{rokintype1} below).  Compared with \cite{JM}, the main new tool is a modification of an averaging trick which has appeared before in studying the zero sets of analytic functions \cite{BB, LR}.  We consider it of independent interest that the technique provides the sharp bound when executed properly in the case when $K$ is a Euclidean ball in all dimensions, and a wide class of convex sets when $d=2$.

\subsection{Acknowledgements}  The author B. Jaye is supported in part by NSF grants DMS-2049477 and DMS-2103534 and M. Mitkovski by NSF DMS-2000236.  This research was primarily carried out while in residence at the ICERM program Harmonic Analysis and Convexity in Fall 2022.  The authors would like to thank Galyna Livshyts and Fedor Nazarov for helpful remarks.

\section{Notation}

For a non-negative integer $k$, let $\omega_k$ denote the volume of the unit ball in $\mathbb{R}^k.$ Recall that $\omega_k = \frac{\pi^k/2}{\Gamma(k/2 +1)}$. For $E\subset \mathbb{R}^d$, we define

\begin{equation*}
    \mathcal{H}^k(E) = \lim_{\delta\to 0^{+}}\inf\Bigl\{\omega_k\sum_{j}r_j^k: E \subset \cup_{j}B(x_j,r_j)  \text{ and } r_j \leq \delta\Bigl\}.
\end{equation*}

Restricting $\mathcal{H}^k$ to a $k-$dimensional plane, $\mathcal{H}^k= m_k$, where $m_k$ is the $k-$dimensional Lebesgue measure. Furthermore, $\mathcal{H}^{d-1}(\mathbb{S}^{d-1}) = d\omega_d$.

Let $K\subset \mathbb{R}^d, d\geq 2$, be an origin symmetric compact convex. We denote by $\mathcal{P}\mathcal{W}_p(K)$ the Paley-Wiener space, consisting of functions in $L^p(\R^d)$ whose distributional Fourier transform is supported in $K$.

As mentioned in the introduction, the (lower) surface density of a set $\Gamma \subset \mathbb{R}^d$ is given as,

\begin{equation*}
       D^{-}(\Gamma) = \liminf_{r\mapsto \infty}\inf_{x\in \mathbb{R}^d}\frac{\mathcal{H}^{d-1}(\Gamma\cap B(x,r))}{m_d(B(x,r))}.
\end{equation*}

We define regular sets (and measures) in the same way as given in \cite{JM}. Let us suppose $\varphi:[0,1)\mapsto [0,\infty)$ is function continuous at 0. We say a measure $\mu$ is $\varphi-$regular if for every $x\in \mathbb{R}^d$ and for all $r\in(0,1),$

\begin{equation*}
    \mu(B(x,r)) \leq \varphi(r)\omega_{d-1}r^{d-1}.
\end{equation*}

We say a closed set $E\subset \mathbb{R}^d$ is $\varphi-$regular if the measure $\mathcal{H}^{d-1}|_{E}$ is $\phi-$regular.

For an origin-symmetric convex set $K$, we denote by $\mathbf{W}(K)$ the mean width of $K$, which is defined by

\begin{equation*}
    \mathbf{W}(K) = \frac{2}{\mathcal{H}^{d-1}(\mathbb{S}^{d-1})}\int_{\mathbb{S}^{d-1}}h_K(\theta)d\mathcal{H}^{d-1}(\theta),
\end{equation*}
where $h_K(\theta) = \max_{x\in K}\langle x,\theta\rangle$ is the support function of $K$.

Clearly, for an origin-centered ball with radius $R$, we have $\mathbf{W}(B(0,R)) = 2R$. It is also not hard to check that $\mathbf{W}([-R,R]^d) = \frac{2R\omega_{d-1}}{\omega_d}$ (see, e.g., Section 5 of \cite{JM}).

\section{The proof of Theorem \ref{Maintheo}}

The key new result we prove is the following:

\begin{prop}\label{rokintype1}
   If $f\in \mathcal{P}\mathcal{W}_{\infty}(K)$ satisfies $\|f\|_{\infty}\leq 1$ and $|f(0)|>0$, then

   \begin{equation}\label{noavedens}
       \limsup_{R\to\infty}\frac{\mathcal{H}^{d-1}(B(0,R)\cap \{f=0\})}{\omega_d R^d} \leq A_d\mathbf{W}(K),
   \end{equation}
where

\begin{equation*}
    A_d = \frac{d}{2}\frac{\omega_d}{\omega_{d-1}}.
\end{equation*}
\end{prop}

The path from Proposition \ref{rokintype1} to Theorem \ref{Maintheo} follows the same lines as in \cite{JM}.  First, we note that (\ref{noavedens}) implies the averaged estimate
   \begin{equation*}
       \limsup_{R\mapsto\infty}\frac{1}{\omega_d R^d}\int_{0}^{R}\mathcal{H}^{d-1}(B(0,r)\cap \{f=0\})\frac{dr}{r} \leq \frac{A_d}{d}\mathbf{W}(K),
   \end{equation*}
for the same value of $A_d$ (this gives an improved version of Proposition 3.1 of \cite{JM}).  Now, following Section 3 of \cite{JM} line-for-line with this new value of constant $A_d$, we first obtain the following improved version of Proposition 3.2 of \cite{JM}:

\begin{prop}\label{compactness1}
 Fix $\delta>0, R_0>0.$ There exists $\varepsilon>0$ such that for every $\varphi-$regular set $\Gamma$ and $f\in \mathcal{P}\mathcal{W}_{\infty}(K)$ that satisfies $\|f\|_{\infty}\leq 1$ and $|f(0)|>1/2$, there exists $R\geq R_0$ such that 

 \begin{equation*}
    \frac{1}{\omega_d R^d}\int_{0}^{R}\mathcal{H}^{d-1}(\Gamma\cap B(x,r)\cap \{|f|\leq \varepsilon\})\frac{dr}{r} \leq \varphi(0)\Bigl(\frac{A_d}{d}\mathbf{W}(K)+\delta\Bigl).
 \end{equation*}

\end{prop}

With this result in hand, one completes the proof of Theorem \ref{Maintheo} in precisely the same manner as in Section 3.4 of \cite{JM}.  We now return to give the proof of Proposition \ref{rokintype1}.


\begin{proof}[Proof of Proposition \ref{rokintype1}]

For $\theta\in \Sph^{d-1}$, $R>0$ and $\eps>0$, consider the quantity
$$
V_{\theta} = \int_{B(0,R)}\int_0^{\eps R} \text{card}(\{|s|\leq t: f(x+s\theta)=0\})\frac{dt}{t}dm_d(x).
$$

For ease of notation,  start by fixing $\theta =(0,0,....,0,1) \in \mathbb{S}^{d-1}$ and consider $x = (x',x_d)\in \R^d$ where $x'\in \mathbb{R}^{d-1}$. Denote by $f_{x'}(t) = f(x',t)$, which has its one-dimensional Fourier transform supported in the interval $[-2\pi h_K(\theta), 2\pi h_K(\theta)]$. Therefore, $f_{x'}$ extends to a entire function in $\mathbb{C}$ and
    $|f_{x'}(t+si)|\leq e^{2\pi h_K(\theta) |s|}$ for $t,s\in \R$.
Applying Jensen's formula, we obtain that for any $x\in \R^d$,
\begin{equation}\nonumber\begin{split}\int_0^{\eps R}  \text{card}(\{|s|\leq t: & \,f_{x'}(x_n+s)=0\})\frac{dt}{t} \\
&\leq  \frac{1}{2\pi}\int_0^{2\pi}2\pi h_K(\theta) \eps R |\sin(\varphi)|d\varphi +\log\Bigl(\frac{1}{|f_{x'}(x_n)|}\Bigl)\\
&= 4\eps Rh_K(\theta)+\log{\bigg(\frac{1}{|f(x)|}\bigg)}.
\end{split}\end{equation}
Substituting this bound into the definition of $V_{\theta}$ yields 
\begin{equation}\label{Vupper}\begin{split}
   V_{\theta}  &\leq   4\eps R^{d+1}\omega_dh_K(\theta)+\int_{B(0,R)} \log{\bigg(\frac{1}{|f(x)|}\bigg)}dm_d(x).
\end{split}\end{equation}

On the other hand, $V_{\theta}$ equals
$$\int\limits_{B^{d-1}(0,R)} \int_{-\sqrt{R^2-|x'|^2}}^{\sqrt{R^2-|x'|^2}}\int_{0}^{\epsilon R}\!\!\!\card(\{|s|\leq t: f_{x'}(x_d+s)=0\})\frac{dt}{t}dm_1(x_d)dm_{d-1}(x').$$
Fix $x'\in B^{(d-1)}(0,R)$, and define a locally finite Borel measure $\mu$ on $\R$ via $$\mu = \sum_{s\in \R\,:\,f_{x'}(s)=0}\delta_s.$$ 
Then
\begin{equation}\begin{split}\nonumber\int_{-\sqrt{R^2-|x'|^2}}^{\sqrt{R^2-|x'|^2}}&\int_{0}^{\epsilon R}\int_{-t}^td\mu(x_d+s)\frac{dt}{t}dm_1(x_d) \\&= \int_{-\sqrt{R^2-|x'|^2}-\epsilon R}^{\sqrt{R^2-|x'|^2}+\epsilon R}\int_{0}^{\epsilon R}\int_{-\sqrt{R^2-|x'|^2}}^{\sqrt{R^2-|x'|^2}}\mathbf{1}_{[r-t,r+t]}(x_d)dm_1(x_d)\frac{dt}{t}d\mu(r).\end{split}\end{equation}
Now, note that if $r\in [-\sqrt{R^2-|x'|^2}+\epsilon R, \sqrt{R^2-|x'|^2}-\epsilon R]$, and $t\in (0, \eps R)$, then $$[r-t, r+t]\subset [-\sqrt{R^2-|x'|^2}, \sqrt{R^2-|x'|^2}],$$
and so
\begin{equation}\begin{split}\int_{-\sqrt{R^2-|x'|^2}}^{\sqrt{R^2-|x'|^2}}&\int_{0}^{\epsilon R}\int_{-t}^td\mu(x_d+s)\frac{dt}{t}dm_1(x_d) \\&\geq 2\eps R\cdot \mu([-\sqrt{R^2-|x'|^2}+\epsilon R, \sqrt{R^2-|x'|^2}-\epsilon R]).\end{split}\end{equation}
Observing that
$$\sqrt{R^2-|x'|^2}-\eps R\geq \sqrt{(1-2\eps)R^2-|x'|^2} \text{ provided that }|x'|\leq \sqrt{1-2\eps}\cdot R,$$
we infer that, with $\theta = (0,0,\dots, 1)$, $V_{\theta}$ is at least
$$2\eps R\cdot\!\!\!\!\!\!\!\!\!\!\int\limits_{B^{(d-1)}(0,\sqrt{1-2\eps}\cdot R)}\!\!\!\!\!\!\!\!\!\!\!\!\!\!\card(\{|r|\leq \sqrt{(1-2\eps)R^2-|x'|^2}: \,f(x'+r)=0\}) dm_{d-1}(x').$$
Therefore, with a suitable rotation, we find that for every $\theta\in \Sph^{d-1}$,
$$V_{\theta}\geq 2\eps R\int\limits_{\theta^{\perp}\cap B(0, \sqrt{1-2\eps}\cdot R)}\card(B(0, \sqrt{1-2\eps}\cdot R)\cap \{f=0\}\cap \ell_{y,\theta}) dm_{d-1}(y),$$
where $\ell_{y, \theta}$ is the line through $y$ with direction $\theta$.  Combining this with (\ref{Vupper}) therefore yields
\begin{equation}\begin{split}\label{beforesphe}
\int\limits_{\theta^{\perp}\cap B(0, \sqrt{1-2\eps}\cdot R)}\!\!\!\!\!\!&\card(B(0, \sqrt{1-2\eps}\cdot R)\cap \{f=0\}\cap \ell_{y,\theta}) dm_{d-1}(y)\\
&\leq 2h_K(\theta)\omega_dR^d+\frac{1}{2\eps R}\int_{B(0,R)}\log{\bigg(\frac{1}{|f(x)|}\bigg)}dm_d(x).
\end{split}\end{equation}

The Crofton formula (e.g. \cite[3.2.26]{F}) states that for any set $E\subset \R^d$ that is $(d-1)$-rectifiable, 
$$\mathcal{H}^{d-1}(E) = \frac{1}{2\omega_{d-1}}\int_{\Sph^{(d-1)}}\int_{\theta^{\perp}}\card(E\cap \ell_{y, \theta})dm_{d-1}(y)d\mathcal{H}^{d-1}(\theta).$$
Whence, integrating (\ref{beforesphe}) over $\Sph^{d-1}$ with respect to the $\mathcal{H}^{d-1}$ measure yields that
\begin{equation}\nonumber\begin{split} & \frac{\mathcal{H}^{(d-1)}(B(0, \sqrt{1-2\eps}R)\cap  \{f=0\})}{\omega_dR^d}\\&\leq \Bigl\{\frac{d\omega_d}{2\omega_{d-1}}\mathbf{W}(K)+\frac{d}{2\omega_{d-1}}\frac{1}{2\eps R^{d+1}}\int_{B(0,R)}\log{\bigg(\frac{1}{|f(x)|}\bigg)}dm_d(x)\Bigl\}.\end{split}\end{equation}

Regarding the second term on the right hand side of this inequality, it follows from work of Ronkin \cite{LR} on functions with completely regular growth (see Lemma 4.4 of \cite{JM} for a concise proof) that 
$$\lim_{R\to \infty}\frac{1}{R^{d+1}}\int_{B(0,R)}\ln{\bigg(\frac{1}{|f(x)|}\bigg)}dm_d(x)=0.$$
Therefore,
$$\limsup_{R\to \infty}\frac{\mathcal{H}^{(d-1)}(B(0, \sqrt{1-2\eps}R)\cap  \{f=0\})}{\omega_d R^d} \leq \frac{d\omega_d}{2\omega_{d-1}}\mathbf{W}(K).
$$
Letting $\eps\to 0$ completes the proof of the proposition.
\end{proof}

\section{The sharpness of the bound}

\subsection{The general construction}  Assume that $K$ is an origin symmetric strictly convex body. For $x\in \partial K$, let $\nu(x)$ be the outward pointing unit normal vector to $K$.

Let us recall the polar body of $K$, $$K^{\circ} = \{y\in \R^d: \langle x,y\rangle \leq 1\text{ for every }x\in K\},$$
which satisfies that
$$\|x\|_{K^{\circ}} : = \inf\{\lambda\geq 0: x\in \lambda K^{\circ}\} = h_K(x).
$$

For a bounded continuous function $g:\partial K\to [0,\infty)$, set
$$\langle g\rangle  = \frac{1}{\Haus^{d-1}(\partial K)}\int_{\partial K}g(x)d\Haus^{d-1}(x)
$$
and
$$\mu = \frac{1}{\Haus^{d-1}(\partial K)}\sup_{y\in K^{\circ}} \int_{\partial K}|\langle \nu(x), y\rangle| g(x) d\Haus^{d-1}(x).$$  Our first goal is to prove the following

\begin{prop}\label{genexample}  For every $\eps>0$ and every bounded continuous function $g:\partial K\to [0,\infty)$, there exist
\begin{enumerate}
\item a function $f\in \mathcal{P}\mathcal{W}_{\infty}(K)$ with $|f(0)|=\|f\|_{\infty}=1$, and 

\item  a continuous function $\varphi:[0,\infty)\to [0,\infty)$ with $\varphi(0)=1$, and a $\varphi$-regular set $\Gamma\subset \R^d$ with $\Gamma\subset \{f=0\}$,
\end{enumerate}
such that 
$$\mathbf{D}^-(\Gamma)\geq \frac{2\langle g\rangle}{\mu}-\eps.$$
\end{prop}

\begin{proof} For $N\in \mathbb{N}$, select $x_1,\dots, x_N$ uniformly and independently on $\partial K$, and consider the associated  vectors $\nu_n = \nu(x_n)$ for $n=1,\dots, N$.  For $\alpha>0$, consider the function
$$f(x) = \prod_{n=1}^N \cos\Bigl(2\pi \frac{\alpha g(x_n)}{N}\langle x, \nu_n\rangle\Bigl).$$
The Fourier transform of $f$ is the $N$-fold convolution of the factors $$\tfrac{1}{2}\bigl(\delta_{\frac{\alpha g(x_n)}{N}\nu_n}+\delta_{-\frac{\alpha g(x_n)}{N}\nu_n}\bigl)$$
and therefore $$\supp(\widehat{f})\subset \Bigl\{\, \frac{\alpha}{N}\sum_{n=1}^N g(x_n)\eps_n\nu_n: \; \eps_n\in \{-1,1\}\Bigl\}.$$
Our goal is to find (the largest) $\alpha>0$ to ensure that $\supp(\wh{f})\subset K$.  Recall that, for any $x\in \R^d$ $\|x\|_K = \sup_{y\in \partial K^{\circ}}\langle x,y\rangle,$
and $x\in K$ if and only if $\|x\|_K\leq 1$.  Therefore, we want to select $\alpha>0$ so that, for any $y\in \partial K^{\circ}$,
$$ \alpha \sup_{\eps_n\in \{-1,1\}}\frac{1}{N}\sum_{n=1}^Ng(x_n)\eps_n\langle \nu_n, y\rangle\leq 1,
$$
or, in other words,
$$\alpha \frac{1}{N}\sum_{n=1}^N|\langle \nu_n, y\rangle |g(x_n)\leq 1.$$
For $y\in \partial K^{\circ}$, the random variable $X_{n,y}=|\langle \nu_n, y\rangle |g(x_n)$ has mean
$$\mu_y = \frac{1}{\Haus^{d-1}(\partial K)}\int_{\partial K}|\langle \mu(x),y\rangle|g(x) d\Haus^{d-1}(x),$$
and its variance may be crudely bounded independently of $n$ in terms of the geometry of the convex body $K$ and the $L^{\infty}(\partial K)$ norm of $g$.  Since, $X_{n,y}$ are independent, Markov's inequality yields that
\begin{equation}\label{weaklaw}\text{Prob}\Bigl(\Bigl|\frac{1}{N}\sum_{n=1}^N|\langle \nu_n, y\rangle|g(x_n) - \mu_y\Bigl|>\delta\Bigl) \leq  \frac{C(K,g)}{N\delta^2}.\end{equation}
Now, for  $\nu_1, \dots, \nu_N$ fixed on $\Sph^{d-1}$, the function $y\mapsto \frac{1}{N}\sum_{n=1}^N |\langle \nu_n, y\rangle|g(x_n)$ is Lipschitz continuous with Lipschitz constant bounded by $K:=\|g\|_{\infty}$. Similarly, the function $\mu_y$ is Lipschitz continuous with constant $\leq K$. 
 Elementary volume considerations ensure that the set $\partial  K^{\circ}$ can be covered by $C(\delta/K)^{-(d-1)}$ balls $B(y_j,\delta/K)$ with $y_j\in \partial K^{\circ}$.  Therefore, if 
$\max_{y\in \partial K^{\circ}}\Bigl|\frac{1}{N}\sum_{n=1}^N|\langle \nu_n, y\rangle|g(x_n) - \mu_y\Bigl|>3\delta$, then there must exist $j$ with $\Bigl|\frac{1}{N}\sum_{n=1}^N|\langle \nu_n, y_j\rangle|g(x_n) - \mu_{y_j}\Bigl|>\delta$.  Consequently, (\ref{weaklaw}) ensures that
\begin{equation}\begin{split}\nonumber\text{Prob}&\Bigl(\max_{y\in \partial K^{\circ}}\Bigl|\frac{1}{N}\sum_{n=1}^N|\langle \nu_n, y\rangle|g(x_n) - \mu_y\Bigl|>3\delta\Bigl) \\&\leq \sum_{j}\text{Prob}\Bigl(\Bigl|\frac{1}{N}\sum_{n=1}^N|\langle \nu_n, y_j\rangle|g(x_n) - \mu_{y_j}\Bigl|>\delta\Bigl)\to 0 \text{ as }N\to \infty. \end{split}\end{equation}

Since $\mu = \sup_{y\in \partial K^{\circ}}\mu_y$, we conclude that for any $\delta>0$, if $N$ is chosen sufficiently large then there exist $x_1,\dots, x_N\in \partial K$ such that for any $\eps_n\in \{-1,1\}$,
\begin{equation}\label{Fsupp}
\Bigl\|\frac{1}{N}\sum_{n=1}^Ng(x_n)\eps_n\nu_n\Bigl\|_K\leq \mu+3\delta,
\end{equation}
and so if $\alpha = \frac{1}{\mu+3\delta}$, then $\supp(\widehat{f})\subset K.$  Repeating the Markov's inequality argument if necessary, we may additionally ensure that
\begin{equation}\begin{split}\nonumber\frac{1}{N}\sum_{n=1}^N g(x_n)\geq \frac{1}{\Haus^{d-1}(\partial K)}\int_{\partial K}g(x)d\Haus^{d-1}(x)-\delta = \langle g\rangle -\delta.
\end{split}\end{equation}

On the other hand, each factor
$f_n(x) = \cos\bigl(2\pi \frac{\alpha}{N}g(x_n)\langle x,\nu_n\rangle\bigl)$
satisfies
$$\lim_{R\to \infty}\frac{\mathcal{H}^{d-1}(\{f_n=0\}\cap B(0,R))}{\omega_dR^d} = \frac{2\alpha g(x_n)}{N}.$$
(See Lemma 5.1 of \cite{JM}.)  Since there are $N$ factors $f_n$, and the nodal sets of each $f_n$ intersect in a set of dimension $d-2$, we have that,
\begin{equation}\begin{split}\nonumber\lim_{R\to \infty}&\frac{\mathcal{H}^{d-1}(\{f=0\}\cap B(0,R))}{\omega_dR^d} = 2\alpha \frac{1}{N}\sum_{n=1}^N g(x_n)\geq \frac{2(\langle g \rangle -\delta)}{(\mu+3\delta)}.\end{split}\end{equation}

We cannot immediately conclude Lemma \ref{genexample} as the set $\{f=0\}$ is not $\varphi$-regular for a function $\varphi$ with $\lim_{t\to 0^+}\varphi(t)=1$.  However by removing small regions where any of the planes in the sets $\{f_n=0\}$ intersect, we obtain a  set  $\Gamma$ that is $\varphi$-regular for some $\varphi$ with $\lim_{t\to 0^+}\varphi(t)=1$, and such that $\Gamma$ has surface density at least $2\frac{\langle g\rangle}{\mu}-C'\delta$, where $C'$ is an absolute constant, and $\{f= 0\}\supset\Gamma$.  This concludes the proof of Proposition \ref{genexample}.\end{proof}

We first use Proposition~\ref{genexample} to prove Theorem~\ref{sharp}.

\begin{proof}[Proof of Theorem \ref{sharp}]
Take $g\equiv 1$ and recall the following well-known computation (see, for instance the end of Section 5 of \cite{JM})
$$ \int_{\Sph^{d-1}}|\langle \theta,v\rangle|d\mathcal{H}^{d-1}(\theta) = 2\omega_{d-1}.$$
Since $\mathbf{W}(B(0,1))=2$, the result immediately follows from Proposition~\ref{genexample}.
\end{proof}

\subsection{Sharpness for any $\pi/2$-symmetric convex body if $d=2$}

We now prove Theorem~\ref{sharp2} as a consequence of Proposition~\ref{genexample}.  Recall that a convex body is called $\pi/2$-symmetric if it is symmetric under a rotation by $\pi/2$, i.e.
$$(x_1,x_2)\in K\iff (-x_2, x_1)\in K$$
This condition implies origin symmetry, and  that $h_K(\theta_1,\theta_2) = h_K(-\theta_2, \theta_1)$ for $(\theta_1,\theta_2)\in \Sph^1$.  

By an approximation argument, in proving Theorem \ref{sharp2}, we may assume that $K$ is strictly convex.  We again will set $g\equiv 1$ in the statement of Proposition \ref{genexample}, and calculate, for $\theta\in \partial K^{\circ}$
$$\frac{1}{\Haus^1(\partial K)}\int_{\partial K}|\langle \nu_x, \theta\rangle| d\Haus^1(x) = \frac{2}{\Haus^1(\partial K)}\int_{(\partial K)_+}\langle \nu_x, \theta\rangle d\Haus^1(x)$$
where $(\partial K)_+ = \{x\in \partial K: \langle \nu_x,\theta\rangle \geq 0\}$.  

Denote by $z_{\pm}\in \partial K$ the two points that satisfy $\nu_{z_{\pm}}\perp \theta$, and put $\theta^{\perp} = (-\theta_2, \theta_1)$.  Since $K$ is origin symmetric, $z_{-} = -z_{+}$, so the line segment $[z_{-}, z_{+}]\subset K$ contains $0$, and has length $2|\nabla h_K(\tfrac{\theta^{\perp}}{|\theta|})|.$  Put $\nu$ to be a unit vector normal to the direction of the line segment $[z_-,z_+]$.  Then the divergence theorem implies that 
$$\int_{(\partial K)_+}\langle \nu_x, \theta\rangle d\Haus^1(x) = \int_{[z_-,z_+]}|\langle \nu, \theta\rangle|d\Haus^1 = 2|\nabla h_K(\tfrac{\theta^{\perp}}{|\theta|})||\langle \nu, \theta\rangle|.
$$
It is an elementary geometry exercise to see that
$$|\nabla h_K(\tfrac{\theta^{\perp}}{|\theta|})||\langle \nu, \theta\rangle| = h_K(\tfrac{\theta^{\perp}}{|\theta|})|\theta|.$$
Indeed, this boils down to the following fact: For $A>B>0$ denote by $T$ the right angle triangle with vertices $(0,0)$, $(0,B)$ and $(\sqrt{A^2-B^2},0)$, then the outward unit vector to the hypotenuse of $T$ has its first component equal $B/A$.

Since
$$h_K(\tfrac{\theta^{\perp}}{|\theta|})|\theta| = h_K(\theta^{\perp})$$
we have that the quantity $\mu$ appearing in Proposition \ref{genexample} is
$$\mu = \sup_{\theta\in \partial K^{\circ}}\frac{4h_K(\theta^{\perp})}{\Haus^1(\partial K)}.
$$
Combined with the classical fact that for two-dimensional convex body $K$,
$$\Haus^1(\partial K ) = \pi \mathbf{W}(K),
$$
Proposition \ref{genexample} ensures that for any $\eps>0$ there is a function $f\in \mathcal{P}\mathcal{W}_{\infty}(K)$, that vanishes on a regular set of density at least $$\frac{\pi}{2}\frac{\mathbf{W}(K)}{\sup_{\theta\in \partial K^{\circ}}h_K(\theta^{\perp})}-\eps.$$

To this point we have not used the $\tfrac{\pi}{2}$-symmetric assumption, but finally we observe that  if $K$ is $\tfrac{\pi}{2}$-symmetric then $\sup_{\theta\in \partial K^{\circ}}h_K(\theta^{\perp})=1$ (insofar as it implies that $h_K(\theta)=h_K(\theta^{\perp})$), and this completes the proof.

\medskip

\vspace{0.3cm}

Benjamin Jaye,
School of Mathematics,
Georgia Institute of Technology,
Atlanta, GA USA 30332

\smallskip
{\it E-mail}: \texttt{bjaye3@gatech.edu}
\vspace{0.3cm}

Mishko Mitkovski
School of Mathematical and Statistical Sciences,
Clemson University,
Clemson, SC 29634

\smallskip
{\it E-mail}: \texttt{mmitkov@clemson.edu}

\vspace{0.3cm}
Manasa N. Vempati,
Department of Mathematics,
Louisiana State University,
Baton Rouge, LA 70803-4918, USA.

\smallskip
{\it E-mail}: \texttt{nvempati@lsu.edu}

\end{document}